\documentclass[12pt,a4paper]{article}

\usepackage{a4wide}

\usepackage[latin1]{inputenc}
\usepackage[T1]{fontenc}
\usepackage[english]{babel}

\usepackage{amsmath}
\usepackage{amssymb}
\usepackage{amsfonts,makeidx,amsthm,mathrsfs}

\newcommand{\Prond}{\mathscr{P}}
\newcommand{\F}{\mathscr{F}}

\newcommand{\Lr}{\mathscr{L}}
\newcommand{\dvol}{\frac{\omega^n}{n!}}
\newcommand{\dvolphi}{\frac{\omega_{\phi}^n}{n!}}
\newcommand{\dsubvol}{\frac{\omega^{n-1}}{(n-1)!}}
\newcommand{\dsubsubvol}{\frac{\omega^{n-2}}{(n-2)!}}

\newcommand{\tr}{\mathrm{tr}}

\newcommand{\h}{\mathfrak{h}}

\newcommand{\R}{\mathbb{R}}
\newcommand{\C}{\mathbb{C}}

\def\ham{\mathop{\mathrm{Ham}}\nolimits}

\def\symp{\mathop{\mathrm{Symp}}\nolimits}

\newcommand{\E}{\mathcal{E}}
\newcommand{\J}{\mathcal{J}}

\newcommand{\LC}{\mathrm{lc}}
\newcommand{\MOm}{\mathcal{M}_{\Theta}}
\newcommand{\grad}{\textrm{grad}}
\newcommand{\D}{\mathcal{D}}

\newcommand{\extwedge}{\stackrel{\circ}{\wedge}}

\newtheorem{theoremintro}{Theorem}
\newtheorem{theoremprinc}{Theorem}
\newtheorem{theorem}{Theorem}[section]

\newtheorem{lemme}[theorem]{Lemma}
\newtheorem{cor}[theorem]{Corollary}
\newtheorem{prop}[theorem]{Proposition}
\newtheorem{defi}[theorem]{Definition}

\theoremstyle{definition}

\theoremstyle{remark}
\newtheorem{obs}[theorem]{Observation}
\newtheorem{rem}[theorem]{Remark}

\begin{document}

\renewcommand{\refname}{Bibliography}

\title{Futaki invariant for Fedosov's star products.\footnote{Work supported by the Belgian Interuniversity Attraction Pole (IAP) DYGEST}}

\author{\small La Fuente-Gravy Laurent \\
	\scriptsize{lfuente@ulg.ac.be}\\
	\footnotesize{D\'epartement de Math\'ematiques, Universit\'e de Li\`ege}\\[-7pt]
	\footnotesize 4000 Li\`ege, Belgium \\[-7pt]} 

\maketitle

\begin{abstract}
We study obstructions to the existence of closed Fedosov's star products on a given K\"ahler manifold $(M,\omega,J)$. In our previous paper \cite{LLF}, we proved that the Levi-Civita connection of a K\"ahler manifold will produce a closed (in the sense of Connes-Flato-Sternheimer \cite{CFS}) Fedosov's star product only if it is a zero of a moment map $\mu$ on the space of symplectic connections. By analogy with the Futaki invariant obstructing the existence of cscK metrics, we build an obstruction for the existence of zero of $\mu$ and hence for the existence of closed Fedosov's star product on a K\"ahler manifold.
\end{abstract}

\noindent {\footnotesize {\bf Keywords:} Symplectic connections, Moment map, Deformation quantization, closed star products, K\"ahler manifolds.}

\newpage

\tableofcontents


\section{Introduction}


In \cite{cagutt}, a moment map $\mu$ on the space of symplectic connections is introduced. The study of zeroes of $\mu$ and of the so-called critical symplectic connections was first proposed by D.J. Fox \cite{Fox} in analogy with the moment map picture for the Hermitian scalar curvature on almost-K\"ahler manifolds. Recently \cite{LLF}, we give additional motivations for the study of $\mu$, and its zeroes on K\"ahler manifolds, coming from the formal deformation quantization of symplectic manifolds.

Our goal is to exhibit an obstruction to the existence of zeroes of $\mu$ on closed K\"ahler manifolds in the spirit of Futaki invariants. We will consider closed K\"ahler manifolds $(M,\omega,J)$ with K\"ahler class $\Theta$ and denote by $\mathfrak{h}$ the Lie algebra of the reduced group of complex automorphisms of the K\"ahler manifold that is the Lie algebra of holomorphic vector fields of the form $Z=X_F+JX_H$, where we denoted by $X_K$ the Hamiltonian vector field defined by $i(X_K)\omega=dK$ for $K\in C^{\infty}(M)$ normalised by $\int_M K\dvol=0$. Our first result is:
\begin{theoremintro} \label{theorprinc:Futinvdef}
Let $(M,\omega,J)$ be a closed K\"ahler manifold with K\"ahler class $\Theta$ and Levi-Civita connection $\nabla$. Let $\mathfrak{h}$ be the Lie algebra of the reduced group of complex automorphisms of the K\"ahler manifold. Then, the map 
\begin{equation*}
\mathcal{F}^{\omega}:\mathfrak{h} \rightarrow \R : Z \mapsto \int_M H\mu(\nabla)\dvol,
\end{equation*}
for $Z=X_F+JX_H$ and $\mu$ is the Cahen-Gutt moment map on $\E(M,\omega)$, is a character that does not depend on the choice of a K\"ahler form in the K\"ahler class $\Theta$.
\end{theoremintro}

Deformation quantization as defined in \cite{BFFLS} is a formal associative deformation of the Poisson algebra $(C^{\infty}(M),.,\{\cdot,\cdot\})$ of a Poisson manifold $(M,\pi)$ in the direction of the Poisson bracket. The deformed algebra is the space $C^{\infty}(M)[[\nu]]$ of formal power series of smooth functions with composition law $*$ called star product. 

On a symplectic manifold $(M,\omega)$ endowed with a symplectic connection $\nabla$ (i.e. torsion-free connection leaving $\omega$ parallel), one can associate the Fedosov's star product $*_{\nabla}$, \cite{fed2}. The moment map $\mu$ evaluated at $\nabla$ is the first non-trivial term in the expression of a trace density for the star product $*_{\nabla}$, see \cite{LLF}. So that, if the star product $*_{\nabla}$ is closed (in the sense of Connes-Flato-Sternheimer \cite{CFS}), then $\mu(\nabla)$ is the zero function which implies the following result.

\begin{cor} \label{cor:corprinc}
Let $(M,\omega,J)$ be a closed K\"ahler manifold with K\"ahler class $\Theta$, such that $\mathcal{F}^{\omega}$ is not identically zero, then, given any  K\"ahler form $\widetilde{\omega} \in \MOm$ with Levi-Civita connection $\widetilde{\nabla}$, the Fedosov's star product $*_{\widetilde{\nabla}}$ is not closed.
\end{cor}

Finally, we identify the character $\mathcal{F}^{\omega}$ with one of the so-called higher Futaki invariants. It enables us to exhibit an example of K\"ahler manifolds \cite{NP,Ono} admitting non-zero values of $\mathcal{F}$ and hence no closed Fedosov's star products as considered in the above Corollary \ref{cor:corprinc}.


\section{The moment map and Fedosov's star products} \label{sect:symplconn}


Consider a closed symplectic manifold $(M,\omega)$ of dimension $2n$. A symplectic connection $\nabla$ on $(M,\omega)$ is a torsion-free connection such that $\nabla \omega=0$. There always exists a symplectic connection on a symplectic manifold and the space $\E(M,\omega)$ of symplectic connections is the affine space
$$\E(M,\omega)=\nabla+ \Gamma(S^3T^*M) \textrm{ for some } \nabla \in \E(M,\omega),$$
where $S^3T^*M:=\{A\in \Lambda^1(M)\otimes End(TM) \ | \ \omega(A(\cdot)\cdot,\cdot) \textrm{ is completely symmetric}\}$. For $A\in S^3T^*M$, we set $\underline{A}(\cdot,\cdot,\cdot)$ for the symmetric $3$-tensor $\omega(A(\cdot)\cdot,\cdot)$.

There is a natural symplectic form on $\E(M,\omega)$. For $A, B \in T_{\nabla} \E(M,\omega)$, seen as element of $\Lambda^1 (M)\otimes End(TM,\omega)$, one defines
$$\Omega^{\E}_{\nabla}(A,B):= \int_M \tr(A\extwedge B)\wedge \dsubvol= -\int_M \Lambda^{kl}\tr(A(e_k)B(e_l))\frac{\omega^n}{n!},$$
where $\extwedge$ is the product on $\Lambda^1 (M)\otimes End(TM,\omega)$ induced by the usual $\wedge$-product on forms and the composition on the endomorphism part, $\Lambda^{kl}$ is defined by $\Lambda^{kl}\omega_{lt}=\delta^k_t$ for $\omega_{lt}:=\omega(e_l,e_t)$ for a frame $\{e_k\}$ of $T_xM$. 
The $2$-form $\Omega^{\E}$ is a symplectic form on $\E(M,\omega)$. 

\begin{rem}
The symplectic form $\Omega^{\E}$ can be written in coordinate as : 
$$\Omega^{\E}_{\nabla}(A,B):=\int_M \Lambda^{i_1j_1}\Lambda^{i_2j_2}\Lambda^{i_3j_3}\underline{A}_{i_1i_2i_3}\underline{B}_{j_1j_2j_3}\frac{\omega^n}{n!},$$
for $A,B\in T_{\nabla}\E(M,\omega)$.
\end{rem}

There is a natural symplectic action of the group of symplectomorphisms on $\E(M,\omega)$. For $\varphi$, a symplectic
diffeomorphism, we define an action
\begin{equation} \label{eq:action}
(\varphi.\nabla)_X Y := \varphi_*(\nabla_{\varphi^{-1}_* X}\varphi^{-1}_* Y),
\end{equation}
for all $X,Y \in TM$ and $\nabla \in \E(M,\omega)$. 

Recall that a Hamiltonian vector field is a vector field $X_F$ for $F\in C^{\infty}(M)$ such that
$i(X_F)\omega=dF.$
We denote by $\ham(M,\omega)$ the group of Hamiltonian diffeomorphisms of the symplectic manifold $(M,\omega)$ with Lie algebra the space $C^{\infty}_0(M)$ of smooth functions $F$ such that $\int_M F \dvol=0$. 

The action defined in Equation (\ref{eq:action}) restricts to an action of the group $\ham(M,\omega)$. Let $X_F$ be a Hamiltonian vector field with $F\in C^{\infty}_0(M)$, the fundamental vector field on $\E(M,\omega)$ associated to this action is
\begin{equation*}
(X_F)^{*\E}(Y)Z:=\frac{d}{dt}|_0\varphi_{-t}^F.\nabla=(\Lr_{X_F}\nabla)(Y)Z=\nabla^2_{(Y,Z)}X_F + R^{\nabla}(X_F,Y)Z,
\end{equation*}
where $R^{\nabla}(U,V)W:=[\nabla_U,\nabla_V]W-\nabla_{[U,V]}W$ is the curvature tensor of $\nabla$.

Denote by $Ric^{\nabla}$ the Ricci tensor of $\nabla$ defined by $Ric^{\nabla}(X,Y):=\tr[V\mapsto R^{\nabla}(V,X)Y]$
for all $X,Y \in TM$. 

\begin{theorem} [Cahen-Gutt \cite{cagutt}] \label{theor:momentE}
The map $\mu:\mathcal{E}(M,\omega) \rightarrow C^{\infty}_0(M)$ defined by
$$\mu(\nabla):=(\nabla^2_{(e_p,e_q)} Ric^{\nabla})(e^{p},e^q) + P(\nabla)-\mu_0$$
where $\{e_k\}$ is a frame of $T_{x}M$ and $\{e^l\}$ is the symplectic dual frame of $\{e_k\}$ (that is $\omega(e_k,e^l)=\delta_k^l$) and $P(\nabla)$ is the function defined by $P(\nabla)\dvol:=\frac{1}{2}\tr(R^{\nabla}(.,.)\extwedge R^{\nabla}(.,.))\wedge \dsubsubvol$ with $\int_M P(\nabla)\dvol=:\mu_0$, is a moment map for the action of $\ham(M,\omega)$ on $\E(M,\omega)$, i.e.
\begin{equation*} \label{eq:momentmu}
\frac{d}{dt}|_{0} \int_M \mu(\nabla+tA)F\dvol=\Omega^{\E}_{\nabla}((X_F)^{*\E},A).
\end{equation*}
\end{theorem}


\noindent In \cite{LLF}, the moment map $\mu$ is related to the notion of trace density for Fedosov's star products. Also, the closedness (closedness in the sense of Connes-Flato-Sternheimer \cite{CFS}) of Fedosov's star product implies $\mu=0$. Let us recall briefly all those notions and results.

A {\bf star product}, as defined in \cite{BFFLS}, on $(M,\omega)$ is a $\R[[\nu]]$-bilinear associative law on the space $C^{\infty}(M)[[\nu]]$ of formal power series
of smooth functions : 
\begin{equation} \nonumber 
*: (C^{\infty}(M)[[\nu]] )^2 \rightarrow C^{\infty}(M)[[\nu]]: (H,K)\mapsto H*K:=\sum_{r=0}^{\infty} \nu^r C_r(H,K) 
\end{equation}
where the $C_r$'s are bidifferential operators null on constants such that for all $H,\, K \in C^{\infty}(M)[[\nu]]$ : 
$C_0(H,K)=HK$ and
$C_1(H,K)-C_1(K,H)=\{H,K\}$.

In \cite{fed2}, Fedosov gave a geometric construction of star products on symplectic manifolds using a symplectic connection $\nabla$ and a formal series of closed $2$-forms $\Omega\in \nu \Omega^2(M)[[\nu]]$. We will only consider Fedosov's star products build with $\Omega=0$ and denote them by $*_{\nabla}$. 

Let $*$ be a star product on a symplectic manifold.
A {\bf trace} for $*$ is a $\R[[\nu]]$-linear map
$$\tr : C^{\infty}(M)[[\nu]] \rightarrow \R[[\nu]],$$
satisfying $\tr(F*H)=\tr(H*F)$ for all $F,H \in C^{\infty}(M)[[\nu]]$. 

Any star product $*$ on a symplectic manifold $(M,\omega)$ admit a trace \cite{fed3,NT,gr}. More precisely, there exists $\kappa \in C^{\infty}(M)[[\nu]]$ such that
\begin{equation*}
\tr(F):= \int_M F\kappa \dvol
\end{equation*}
for all $F \in C^{\infty}(M)[[\nu]]$. The function $\kappa$ is called a {\bf trace density}. Moreover, any two traces for $*$ differ from each other by multiplication with a formal constant $C \in \R[\nu^{-1},\nu]]$.

A star product is called {\bf closed} \cite{CFS} if the map $F\mapsto \int_M F \dvol$ satisfies the trace property:
\begin{equation*}
\int_M F*H \dvol = \int_M H*F \dvol, \  \textrm{ for all } F,H \in C^{\infty}(M)[[\nu]].
\end{equation*}

In \cite{LLF}, we linked the moment map with the trace density $\kappa^{\nabla}$ of the Fedosov's star product $*_{\nabla}$ by the formula :
$$\kappa^{\nabla}:= 1+\frac{\nu^2}{24}\mu(\nabla)+O(\nu^3).$$
So that, if $*_{\nabla}$ is closed, then $\mu(\nabla)=0$.


\section{Futaki invariant for $\mu$}

\subsection{Definition and main Theorem}

We consider a closed K\"ahler manifold $(M,\omega,J)$. Let $\Theta$ be the K\"ahler class of $\omega$ and denote
by $\MOm$ the set of K\"ahler forms in the class $\Theta$. By the classical $dd^c$-lemma,
$$\MOm:= \{\omega_{\phi}=\omega+dd^c\phi \textrm{ s.t. } \phi\in C^{\infty}_0(M),\ \omega_{\phi}(\cdot,J\cdot) \textrm{ is positive definite }\},$$
where $d^cF:=-dF\circ J$ for $F\in C^{\infty}(M)$.

Consider the functional 
\begin{equation*}
\omega_{\phi} \in \MOm \mapsto \mu^{\phi}(\nabla^{\phi})\in C^{\infty}(M),
\end{equation*}
where $\mu^{\phi}$ is the moment map on $\E(M,\omega_{\phi})$ and $\nabla^{\phi}$ is the Levi-Civita connection of $g_\phi(\cdot,\cdot):= \omega_{\phi}(\cdot,J\cdot)$. Using the second Bianchi identity, one can write :
$$\mu^{\phi}(\nabla^{\phi})=-\frac{1}{2}\Delta^{\phi} Scal^{\nabla^{\phi}}  + P(\nabla^{\phi})-\mu_0.$$
Note that $\mu_0$ does not depend on $\phi$ and that $\mu^{\phi}(\nabla^{\phi})$ is normalised with respect to the integral using the K\"ahler form $\omega_{\phi}$.

Let $\mathfrak{h}$ be the Lie algebra of all holomorphic vector fields having at least one zero on $M$. For any $Z\in \mathfrak{h}$ and $\omega_{\phi}\in \MOm$, there are unique $F^{\phi},H^{\phi} \in C^{\infty}(M)$ (depending on $\omega_{\phi}$) whose integral with respect to $\dvolphi$ is zero such that $Z=X^{\omega_{\phi}}_{F^{\phi}}+JX^{\omega_{\phi}}_{H^{\phi}}$ where $X^{\omega_{\phi}}_K$ denotes the Hamiltonian vector field of $K\in C^{\infty}(M)$ with respect to the symplectic form $\omega_{\phi}$. 

\begin{defi}
For $\omega_{\phi} \in \MOm$, we define the map
\begin{equation} \label{eq:Futakiinv}
\mathcal{F}^{\omega^{\phi}}:\mathfrak{h}(M) \mapsto \R: Z \mapsto \int_M H^{\phi}\mu^{\phi}(\nabla^{\phi}) \dvolphi,
\end{equation}
for $Z=X^{\omega_{\phi}}_{F^{\phi}}+JX^{\omega_{\phi}}_{H^{\phi}}$ as above.
\end{defi}

Though, the definition of $\mathcal{F}^{\omega_{\phi}}$ seems a priori to depend on the choice of a point in $\MOm$, we will prove it is not the case.

\begin{theoremprinc}
Let $(M,\omega,J)$ be a closed K\"ahler manifold with K\"ahler class $\Theta$ and Levi-Civita connection $\nabla$. Let $\mathfrak{h}$ be the Lie algebra of the reduced group of complex automorphisms of the K\"ahler manifold. Then, the map 
\begin{equation*}
\mathcal{F}^{\omega}:\mathfrak{h}(M) \rightarrow \R : Z \mapsto \int_M H\mu(\nabla)\dvol,
\end{equation*}
for $Z=X_F+JX_H$ and $\mu$ is the Cahen-Gutt moment map on $\E(M,\omega)$, is a character that does not depend on the choice of a K\"ahler form in the K\"ahler class $\Theta$.
\end{theoremprinc}

The Theorem \ref{theorprinc:Futinvdef} implies that the non-vanishing of $\mathcal{F}^{\omega}$ is an obstruction to the existence of $\omega_{\phi}\in \MOm$ such that $\mu^{\phi}(\nabla^{\phi})=0$. 

\begin{proof}[Proof of Corollary \ref{cor:corprinc}]
For $\widetilde{\omega}\in \MOm$ with Levi-Civita connection $\widetilde{\nabla}$, assume the Fedosov's star product $*_{\widetilde{\nabla}}$ is closed. Then $\mu^{\widetilde{\omega}}(\widetilde{\nabla})=0$ and hence $\mathcal{F}^{\omega}=0$. It concludes the proof.
\end{proof}


\subsection{The space $\J_{int}(M,\omega)$}


The goal of this subsection is to state the necessary formulas coming from \cite{LLF} in order to prove Theorem \ref{theorprinc:Futinvdef}.

\begin{defi}
We denote by $\J_{int}(M,\omega)$ the space of integrable complex structures on $M$ compatible with $\omega$, that is $J\in \J_{int}(M,\omega)$ is a complex structure such that $\omega(J.,J.)=\omega(.,.)$ and $\omega(.,J.)$ is a Riemannian metric.
\end{defi}

For $J_t\in \J_{int}(M,\omega)$ a smooth path and $A:=\frac{d}{dt}|_0J_t\in T_J \J_{int}(M,\omega)$. Then, $A$ is a section of the bundle $End(TM)$ satisfying $ AJ_0+J_0A=0$ and the $2$-tensor
$$J(\nabla A (X)Y)-(\nabla A)(JX)Y$$
is symmetric in $X, Y$.

Consider the map 
$$\LC:\J_{int}(M,\omega) \rightarrow \E(M,\omega): J \mapsto \nabla^{J}$$
which associates to an integrable complex structure $J$ compatible with $\omega$, the Levi-Civita 
connection $\nabla^{J}$ of the K\"ahler metric $g_J(.,.):= \omega(.,J.)$. 

The map $\LC$ is equivariant with respect to the group of symplectic diffeomorphisms of $(M,\omega)$.
That is : for all $\varphi \in \symp(M,\omega)$ and $J\in \J_{int}(M,\omega)$ with $\varphi.J:=\varphi_*J\varphi^{-1}_*$:
$$\LC(\varphi.J)=\varphi.\LC(J) .$$

\begin{prop} \label{prop:LC*}
Let $A\in T_J \J_{int}(M,\omega)$ and write $B\in T_{\nabla}\E(M,\omega)$ such that $B=\LC_{*J}(A)$.
Then $B$ is the unique solution to the equation
\begin{equation*} \label{eq:BJlin}
B(X)Y+JB(X)JY=- \nabla J A(X)Y.
\end{equation*}
and if $JA \in T_J \J_{int}(M,\omega)$, then :
$$\LC_{*J}(JA)(X)Y=JB(JX)JY +\frac{1}{2}\left(J(\nabla A) (JX)Y)+(\nabla A)(X)Y\right).$$
\end{prop}

\noindent From those equations we obtain \cite{LLF}:

\begin{lemme} \label{lemme:Iinv}
If $A, A'$ and $JA, JA'\in T_J\J_{int}(M,\omega)$ then
$$(lc^*\Omega^{\E})_J(JA,JA')=(lc^*\Omega^{\E})_J(A,A').$$
\end{lemme}


\subsection{Proof of Theorem \ref{theorprinc:Futinvdef}} \label{sect:cal}

We will prove Theorem \ref{theorprinc:Futinvdef} in this section. For this, consider a smooth one-parameter family $\phi:\ ]-\epsilon,\epsilon[\rightarrow C^{\infty}_0(M):t\mapsto \phi(t)$ for some $\epsilon\in \R^+_0$ such that the $2$-form $\omega_{\phi(t)}:= \omega+ dd^c\phi(t)$ is a smooth path in $\MOm$ passing through $\omega$. To prove the independence of $\mathcal{F}^{\omega^{\phi}}$, we will show that for all $Z\in \mathfrak{h}(M)$ :
\begin{equation*}
\frac{d}{dt}|_{0}\mathcal{F}^{\omega^{\phi(t)}}(Z)=0.
\end{equation*}

All the forms $\omega_{\phi(t)}$ are symplectomorphic to each other. Indeed, set $X_t:=-\grad^{\phi(t)}(\dot{\phi})$ the gradient vector field of $\dot{\phi}(t)$ with respect to $g_{\phi(t)}$ (that is $g_{\phi(t)}(\grad^{\phi(t)}(\dot{\phi}),\cdot)=d\dot{\phi}$). Then the one parameter family of diffeomorphisms $f_t$ integrating the time-dependent vector field $X_t$ satisfies
\begin{equation} \label{eq:Moser}
f_t^*\omega_{\phi(t)}=\omega.
\end{equation}

Consider $f_t$ as in the above equation (\ref{eq:Moser}). Then, the natural action of $f_t^{-1}$ on $J$ produces a path
$$J_t:=f_t^{-1}.J:= f_{t*}^{-1} J f_{t*} \in \J_{int}(M,\omega).$$
Define the associated K\"ahler metric $g_{J_t}(\cdot,\cdot):=\omega(\cdot,J_t\cdot)$ and denote by $\nabla^{J_t}$ its
Levi-Civita connection. Then, $\nabla^{J_t}$ and $\nabla^{\phi(t)}$ are related by the following formula : 
$$\nabla^{J_t}=f_t^{-1}.\nabla^{\phi(t)},$$
where $(f_t^{-1}.\nabla^{\phi(t)})_YZ = f_{t*}^{-1}\nabla^{\phi(t)}_{f_{t*}Y}f_{t*}Z$. Then, their image by the moment map is related by :
\begin{equation} \label{eq:muJmuphi}
\mu(\nabla^{J_t})=f_t^*\mu^{\phi(t)}(\nabla^{\phi(t)}).
\end{equation}
Note that on the LHS the moment map is taken with respect to a fixed sympelctic form while on the RHS $\mu^{\phi(t)}$ is a function on $\E(M,\omega_{\phi(t)})$.

\begin{proof}[Proof of Theorem \ref{theorprinc:Futinvdef}]
We will use the notations introduced above. First, using Equations (\ref{eq:Futakiinv}), (\ref{eq:Moser}) and (\ref{eq:muJmuphi}), we have :
$$\mathcal{F}^{\omega^{\phi(t)}}(Z)=\int_M H^{\phi(t)}\mu^{\phi(t)}(\nabla^{\phi(t)}) \frac{\omega_{\phi(t)}^n}{n!}=\int_M f_t^*(H^{\phi(t)})\mu(\nabla^{J_t})\dvol.$$
We will differentiate at $t=0$. We will write $H$ for $H^{\phi(0)}$ :
$$\frac{d}{dt}|_0 f_t^*(H^{\phi(t)}) = X_0(H)+\frac{d}{dt}|_0 H^{\phi(t)}=X_0(H)+Z(\dot{\phi}(0))=X_F(\dot{\phi}(0)).$$

Using the fact that $\frac{d}{dt}|_0J_t=\Lr_{X_0} J=-\Lr_{JX_{\dot{\phi}(0)}}J=-J\Lr_{X_{\dot{\phi}(0)}}J$, we compute
\begin{eqnarray*}
 \frac{d}{dt}|_0\mathcal{F}^{\omega^{\phi(t)}}(Z) & = & \int_M X_F(\dot{\phi}(0))\mu(\nabla) \dvol + \frac{d}{dt}|_0 \int_M H \mu(\nabla^{J_t})\dvol, \\
 & = & \int_M X_F(\dot{\phi}(0))\mu(\nabla) \dvol + \Omega(\Lr_{X_H}\nabla, \LC_{*J}(-J\Lr_{X_{\dot{\phi}(0)}}J)).
\end{eqnarray*}
Using the equivariance of the map $\LC$ and Lemma \ref{lemme:Iinv}, we have
\begin{eqnarray*}
 \Omega(\Lr_{X_H}\nabla, \LC_{*J}(-J\Lr_{X_{\dot{\phi}(0)}}J)) & = & \Omega(\LC_{*J}(\Lr_{X_H}J),\LC_{*J}(-J\Lr_{X_{\dot{\phi}(0)}}J)),\\
& = & \Omega(\LC_{*J}(J\Lr_{X_H}J),\LC_{*J}(\Lr_{X_{\dot{\phi}(0)}}J)).
\end{eqnarray*}
Finally, as $Z$ is holomorphic, $J\Lr_{X_H}J=-\Lr_{X_F}J$, so that :
\begin{eqnarray*}
 \frac{d}{dt}|_0\mathcal{F}^{\omega^{\phi(t)}}(Z) & = & \int_M X_F(\dot{\phi}(0))\mu(\nabla) \dvol - \Omega(\LC_{*J}(\Lr_{X_F}J),\LC_{*J}(\Lr_{X_{\dot{\phi}(0)}}J)),\\
 & = & \int_M X_F(\dot{\phi}(0))\mu(\nabla) \dvol+ \int_M \mu(\nabla) X_{\dot{\phi}(0)}(F)\dvol =0.
\end{eqnarray*}

The fact that $\mathcal{F}$ is a character, is a consequence of the above computations. Indeed, for $Y,Z \in \mathfrak{h}(M)$, one has $[Y,Z]=\frac{d}{dt}|_0\varphi^Y_{-t*}Z$, for $\varphi^Y_{t*}$ the flow of $Y$. Then, when $Z=X^{\omega}_F+JX^{\omega}_H$, one computes $\varphi^Y_{-t*}Z=X^{\varphi^{Y*}_{t}\omega}_{\varphi^{Y*}_{t}F}+JX^{\varphi^{Y*}_{t}\omega}_{\varphi^{Y*}_{t}H}$. Finally, one has
$$\mathcal{F}^{\omega}([Y,Z])=\frac{d}{dt}|_0\mathcal{F}^{\omega}(\varphi^Y_{-t*}Z) = \frac{d}{dt}|_0\mathcal{F}^{\varphi_t^{Y*}\omega}(\varphi^Y_{-t*}Z)= \frac{d}{dt}|_0\mathcal{F}^{\omega}(Z)=0$$
\end{proof}


\section{Generalised Futaki invariants}

\subsection{$\mathcal{F}^{\omega}$ is a generalised Futaki invariant}

In \cite{Fut}, Futaki generalised the Futaki invariant obstructing the existence of K\"ahler-Einstein metrics. One of these so-called generalised Futaki invariants is the invariant we define using the moment map. 

Futaki's construction goes as follows. On a K\"ahler manifold $(M,\omega,J)$, consider the holomorphic bundle $T^{(1,0)}M$ of tangent vectors of type $(1,0)$. Choose any $(1,0)$-connection $\overline{\nabla}$ on $T^{(1,0)}M$ with curvature $R^{\overline{\nabla}}$. For $Z\in \h(M)$, define $L(Z^{(1,0)}):=\overline{\nabla}_{Z^{(1,0)}}-\Lr_{Z^{(1,0)}}$, it is a section of the bundle $End(T^{(1,0)}M)$. Let $q$ be a $Gl(n,\C)$-invariant polynomials on $\mathfrak{gl}(n,\C)$ of degree $p$, Futaki defined in \cite{Fut}, the map $F_q:\h\rightarrow \C$ by
\begin{equation*} \label{eq:Futakigendef}
F_q(Z):= (n-p+1) \int_M u_Z q(R^{\overline{\nabla}})\wedge \omega^{(n-p)}+\int_M q(L(Z^{(1,0)})+R^{\overline{\nabla}})\wedge \omega^{(n-p+1)},
\end{equation*}
where $u_Z$ is the complex valued function defined by $i(Z^{(1,0)})\omega=\overline{\partial} u_Z$.

Futaki shows that $F_q$ depends neither on the choice of the ${(1,0)}$-connection nor on the choice of the K\"ahler form in $\MOm$, see \cite{Fut}. Moreover, if you take $q=c_k$ the polynomials defining the $k$-th Chern form, it is proved in \cite{Fut} that one recovers Bando's obstruction to the harmonicity of the k$^{th}$ Chern form:
\begin{equation} \label{eq:Fc2}
F_{c_k}(Z)=(n-k+1) \int_M u_Z c_k(R^{\nabla})\wedge \omega^{(n-k)}.
\end{equation}

\begin{prop} \label{prop:FequalF}
We have that $\mathcal{F}^{\omega}$ is the imaginary part of $F_{\frac{8\pi^2}{(n-1)!}(c_2-\frac{1}{2}c_1. c_1)}$
\end{prop}

\begin{proof}
The key of the computation is that the Pontryagin $4$-form defining $P(\nabla)$ satisfies:
\begin{equation*}
\tr(R^{\nabla}\extwedge R^{\nabla})=16\pi^2(c_2-\frac{1}{2}c_1. c_1)(R^{\nabla}).
\end{equation*}
Then,
$$\mathcal{F}^{\omega}(Z)=-\frac{1}{2} \int_M H\Delta Scal \dvol +8\pi^2\int_M Hc_2(R^{\nabla})\wedge \dsubsubvol - 4\pi^2\int_M Hc_1.c_1(R^{\nabla}) \wedge \dsubsubvol.$$
As $u_z=F+iH$, Equation (\ref{eq:Fc2}) tells us that the imaginary part of $F_{\frac{8\pi^2}{(n-1)!}c_2}(Z)$ is:
$$8\pi^2\int_M Hc_2(R^{\nabla})\wedge \dsubsubvol.$$  It remains to compute de second term of $F_{\frac{4\pi^2}{(n-1)!}c_1c_1}$ :
\begin{eqnarray*}
F_{\frac{4\pi^2}{(n-1)!}c_1c_1}(Z)& = &4\pi^2\left(\int_M u_Zc_1.c_1(R^{\nabla})\wedge \dsubsubvol+\int_M c_1.c_1(L(Z^{(1,0)})+R^{\nabla})\wedge \dsubvol\right)\\
& = & 4\pi^2\int_M u_Zc_1.c_1(R^{\nabla})\wedge \dsubsubvol+ 2i\int_M \tr^{\C}(L(Z^{(1,0)}))\rho^{\nabla}\wedge \dsubvol.
\end{eqnarray*}
Since $\tr^{\C}(L(Z^{(1,0)}))= \frac{-i}{2}\left(\Delta F+i\Delta H\right)$, we have:
\begin{eqnarray*}
2i\int_M \tr^{\C}(L(Z^{(1,0)}))\rho^{\nabla}\wedge \dsubvol & = & \frac{1}{2}\int_M \left(\Delta F+i\Delta H\right)Scal^{\nabla} \dvol,\\
 & = & \frac{1}{2}\int_M \left( F+iH\right)\Delta Scal^{\nabla} \dvol.
\end{eqnarray*}
So, $\mathcal{F}^{\omega}$ is the imaginary part of $F_{\frac{8\pi^2}{(n-1)!}(c_2-\frac{1}{2}c_1. c_1)}$.
\end{proof}

\subsection{Example}

The computations of generalized Futaki invariants $F_q$ defined by (\ref{eq:Futakigendef}) is not an easy task. For $q=\mathrm{Td}_p$, the invariant polynomials defining the $p^{\mathrm{th}}$ Todd class, methods coming from algebraic geometry are developped to compute $F_{\mathrm{Td}_p}$, see \cite{VedZud,Ono}, in order to study the asymptotic semi-stability \cite{Fut} of the manifold. Those methods and this notion of asymptotic semi-stability are beyond the scope of this paper. However, when the manifold is K\"ahler-Einstein, as it is the case in \cite{Ono}, $F_{\mathrm{Td}_2}$ determines completely $\mathcal{F}^{\omega}$.

\begin{obs} \label{obs}
When $(M,\omega,J)$ is K\"ahler-Einstein, $\mathcal{F}^{\omega}$ is the imaginary part of $\frac{8\pi^2}{(n-1)!}F_{\mathrm{Td}_2}$.
\end{obs}

\begin{proof}
Recall that $\mathrm{Td}_2=c_2+c_1.c_1$. Now, because the Ricci form $\rho=\lambda \omega$, from the computations in the proof of Proposition \ref{prop:FequalF}, one has $F_{c_1.c_1}=0$. So that, $\frac{8\pi^2}{(n-1)!}F_{\mathrm{Td}_2}=F_{\frac{8\pi^2}{(n-1)!}(c_2-\frac{1}{2}c_1. c_1)}$ and its imaginary part is $\mathcal{F}^{\omega}$ by Proposition \ref{prop:FequalF}.
\end{proof}

In \cite{NP}, a 7-dimensional (complex dimension) smooth K\"ahler manifold $(V,\omega,J)$ is constructed, the so-called Nill-Paffenholz example. $V$ is a toric Fano manifold that is K\"ahler-Einstein, \cite{NP}. Moreover, Ono, Sano and Yotsutani \cite{Ono} showed that, on $V$, $F_{\mathrm{Td}_p} \ne 0$ for $2\leq p \leq 7$. Combined with the above Observation \ref{obs}, it means $\mathcal{F}^{\omega}\ne 0$. Consequently, Corollary \ref{cor:corprinc} implies:

\begin{theorem}
Let $(V,\omega,J)$ be the Nill-Paffenholz example \cite{NP} and $\Theta=[\omega]$, then there is no closed Fedosov's star products of the form $*_{\widetilde{\nabla}}$ for $\widetilde{\nabla}$ the Levi-Civita connection of some $\widetilde{\omega}\in \MOm$.
\end{theorem}

\end{document}